\documentclass{article}
\usepackage{amssymb}
\usepackage{verbatim}
\usepackage{array}
\usepackage{latexsym}
\usepackage{enumerate}
\usepackage{amsmath}
\usepackage{amsfonts}
\usepackage{amsthm}
\usepackage{color}
\usepackage[english]{babel}

\newtheorem{theorem}{Theorem}[section]
\newtheorem{proposition}[theorem]{Proposition}
\newtheorem{lemma}[theorem]{Lemma}

\newtheorem{rem}[theorem]{Remark}

\def\cC{\mathcal C}

\def\cX{\mathcal X}
\def\cY{\mathcal Y}
\def\cZ{\mathcal Z}

\def\K{\mathbb{K}}

\def\ord{\mbox{\rm ord}}
\def\deg{\mbox{\rm deg}}

\def\gg{\mathfrak{g}}

\newcommand{\aut}{\mbox{\rm Aut}}

\title{Large $3$-groups of automorphisms of algebraic curves in characteristic $3$}

\author{Massimo Giulietti and G\'abor Korchm\'aros}
\begin{document}
\maketitle


    \begin{abstract}
Let $S$ be a $p$-subgroup of the $\K$-automorphism group $\aut(\cX)$ of an algebraic curve $\cX$ of genus $\gg\ge 2$ and $p$-rank $\gamma$ defined over an algebraically closed field $\mathbb{K}$ of characteristic $p\geq 3$.
In this paper we prove that if $|S|>2(\gg-1)$ then one of the following cases occurs.
\begin{itemize}
\item[(i)] $\gamma=0$ and the extension $\K(\cX)/\K(\cX)^S$  completely ramifies at a unique place, and does not ramify elsewhere.
\item[(ii)] $\gamma>0$, $p=3$, $\cX$ is a general curve, $S$ attains the Nakajima's upper bound $3(\gamma-1)$ and $\K(\cX)$ is an unramified Galois extension of the function field of
a general curve of genus $2$ with equation $Y^2=cX^6+X^4+X^2+1$ where $c\in\K^*$.
\end{itemize}
Case (i) was investigated by Stichtenoth, Lehr, Matignon, and Rocher, see  \cite{stichtenoth1973I,lehr-matignon2005,matignon-rocher2008,rocher1,rocher2}.
\end{abstract}

    \section{Introduction}
In the present paper, $\K$ is an algebraically closed field of
characteristic $p\geq 3$,  $\cX$ is {a} (projective, non-singular,
geometrically irreducible, algebraic) curve of genus $\gg(\cX)\geq
2$ with function field $\mathbb K(\cX)$, $\aut(\cX)$ is the $\K$-automorphism group of $\cX$, and $S$
is a (non-trivial) subgroup of $\aut(\cX)$ whose order is a power of $p.$

The earliest results on the maximum size of $S$ date back to the 1970's
and have played an important role in the study of curves with large automorphism groups exceeding the classical Hurwitz bound $84(\gg(\cX)-1)$. Stichtenoth proved that if $S$ fixes a place $\mathcal P$ of $\mathbb K(\cX)$ then
\begin{equation}
\label{eq117feb2013}
|S|\leq \frac{p}{p-1}\,\gg(\cX)
\end{equation}
unless the extension $\mathbb K(\cX):\mathbb K(\cX)^S$ completely ramifies at $\mathcal P$, and does not ramify elsewhere; in geometric terms, $S$ fixes a point $P$ of $\cX$ and acts on $\cX\setminus\{P\}$ as a semiregular permutation group; see \cite{stichtenoth1973I} and also \cite[Theorem 11.78]{hirschfeld-korchmaros-torres2008}.
In the latter case, the Stichtenoth bound is
\begin{equation}
\label{sti16feb2013}
|S|\leq \frac{4p}{p-1}\,\gg(\cX)^2.
\end{equation}
In his 1987 paper \cite{nakajima1987} Nakajima pointed out that the maximum size of $S$ is related to the Hasse-Witt invariant $\gamma(\cX)$ of $\cX$. It is known
that $\gamma(\cX)$ coincides with the $p$-rank  of $\cX$ defined to be the rank of the (elementary abelian) group of the $p$-torsion points in the Jacobian variety of $\cX$; moreover,   $\gamma(\cX)\leq {\gg}(\cX)$ and when equality holds then $\cX$ is called an {\em ordinary} (or {\em general}) curve; see \cite[Section 6.7]{hirschfeld-korchmaros-torres2008}.
If $S$ fixes a point and (\ref{eq117feb2013}) fails then $\gamma(\cX)=0$; conversely, if $\gamma(\cX)=0$, then $S$ fixes a point, see  \cite[Lemma 11.129]{hirschfeld-korchmaros-torres2008}.
For $\gamma(\cX)>0$, Nakajima proved that $|S|$ divides $\gg(\cX)-1$ when $\gamma(\cX)=1$, and
$|S|\le p/(p-2) (\gamma(\cX)-1)$ otherwise; see \cite{nakajima1987} and also \cite[Theorem 11.84]{hirschfeld-korchmaros-torres2008}. Therefore,
the Nakajima bound \cite[Theorem 1]{nakajima1987} is
\begin{equation}
\label{naka16feb2013}
|S|\leq \left\{
\begin{array}{lll}
{p}/({p-2})\,(\gg(\cX)-1)\quad {\mbox{for}}\quad \gamma(\cX)>1,\\
\gg(\cX)-1\quad {\mbox{for}}\quad \gamma(\cX)=1.
\end{array}
\right.
\end{equation}

In this context, a major  issue is to determine the possibilities for $\cX$, $\gg$ and $S$ when either $|S|$ is close to the Stichtenoth bound (\ref{sti16feb2013}), or $|S|$ is close to the Nakajima bound (\ref{naka16feb2013}).

Lehr and Matignon \cite{lehr-matignon2005} investigated the case where $S$ fixes a point and were able to determine all curves $\cX$ with
\begin{equation}
\label{lehrm} |S|>\frac{4}{(p-1)^2}\,g^2,
\end{equation}
proving that (\ref{lehrm})
only occurs when the curve is birationally equivalent over $\K$ to
an Artin-Schreier curve of equation $Y^q-Y=f(X)$ such that
$f(X)=XS(X)+cX$ where $S(X)$ is an additive polynomial of $\K[X]$.
Later on, Matignon and Rocher \cite{matignon-rocher2008} showed
that the action of a $p$-subgroup of $\K$-automorphisms
$S$ satisfying
$$|S|>\frac{4}{(p^2-1)^2}\,g^2,$$ corresponds to the
\'etale cover of the affine line with Galois group
$S\cong(\mathbb{Z}/p \mathbb{Z})^n$ for $n\leq 3$. These results
have been refined by Rocher, see \cite{rocher1} and
\cite{rocher2}. The essential tools used in the above
mentioned  papers are ramification theory and some structure
theorems about finite $p$-groups.

Curves close to the Nakajima bound are investigated in this paper. The main result is stated in the following theorem.
\begin{theorem}
\label{princ}
Let $S$ be a $p$-subgroup of the $\K$-automorphism group $\aut(\cX)$ of an algebraic curve $\cX$ of genus $\gg(\cX)\geq 2$ defined over an algebraically closed field $\mathbb{K}$ of characteristic $p\geq 3$. If
\begin{equation}
\label{hyp}
|S|>2(\gg(\cX)-1)
\end{equation}
 then one of the following cases occurs:
 \begin{itemize}
\item[{\rm{(i)}}] $\gamma=0$ and the extension $\K(\cX)/\K(\cX)^S$  completely ramifies at a unique place, and does not ramify elsewhere.
\item[{\rm{(ii)}}] $\gamma>0$, $p=3$, $\cX$ is a general curve, $S$ attains the Nakajima's upper bound $3(\gamma-1)$ and $\K(\cX)$ is an unramified Galois extension of the function field of
a general curve of genus $2$ with equation $Y^2=cX^6+X^4+X^2+1$ where $c\in\K^*$.
\end{itemize}
\end{theorem}

The analogous problem for $2$-groups of automorphisms $S$ in characteristic $p=2$ was investigated in \cite{gkp=2}. 

One may also ask how the above results may be refined when $\aut(\cX)$ is larger than $S$. So far, this problem has been
investigated for zero $p$-rank curves $\cX$ such that $\aut(\cX)$ fixes no point of $\cX$; see  \cite{gktrans,gklondon,gur}.

\section{Background and Preliminary Results}\label{sec2}
Let $\bar \cX$ be a non-singular model of $\K(\cX)^S$, that is,
a projective non-singular geometrically irreducible algebraic
curve with function field $\K(\cX)^S$, where $\K(\cX)^S$ consists of all elements of $\K(\cX)$
fixed by every element in $S$. Usually, $\bar \cX$ is called the
quotient curve of $\cX$ by $S$ and denoted by $\cX/S$. The
 field extension
$\K(\cX)/\K(\cX)^S$ is  Galois of degree $|S|$.

Let $\bar{P_1},\ldots,\bar{P_k}$ be the points of the quotient curve $\bar{\cX}=\cX/S$ where the cover $\cX\mapsto\bar{\cX}$ ramifies. For $1\leq i\leq k$, let $L_i$ denote the set of points of $\cX$ which lie {over} $\bar{P_i}$.
In other words, $L_1,\ldots,L_k$ are the short orbits of
$S$ on its faithful action on $\cX$. Here the orbit of $P\in \cX$
$$o(P)=\{Q\mid Q=P^g,\, g\in S\}$$ is {\em long} if $|o(P)|=|S|$, otherwise $o(P)$ is {\em short}. It may be that $S$ has no short orbits. This is the case if and only if every non-trivial element in $S$ is fixed--point-free on $\cX$. On the other hand, $S$ has a finite number of short orbits.

If $P$ is a point of $\cX$, the stabilizer $S_P$ of $P$ in $S$ is
the subgroup of $S$ consisting of all elements fixing $P$. For a
non-negative integer $i$, the $i$-th ramification group of $\cX$
at $P$ is denoted by $S_P^{(i)}$ (or $S_i(P)$ as in \cite[Chapter
IV]{serre1979})  and defined to be
$$S_P^{(i)}=\{g\mid \ord_P(g(t)-t)\geq i+1, g\in
S_P\}, $$ where $t$ is a uniformizing element (local parameter) at
$P$. Here $S_P^{(0)}=S_P^{(1)}=S_P$.

Let $\bar{\gg}$ be the genus of the quotient curve $\bar{\cX}=\cX/S$. The Hurwitz
genus formula  gives the following equation
    \begin{equation}
    \label{eq1}
2\gg-2=|S|(2\bar{\gg}-2)+\sum_{P\in \cX} d_P.
    \end{equation}
    where
\begin{equation}
\label{eq1bis}
d_P= \sum_{i\geq 0}(|S_P^{(i)}|-1).
\end{equation}

Let $\gamma$ be the $p$-rank of $\cX$, and let $\bar{\gamma}$ be the $p$-rank of the quotient curve $\bar{\cX}=\cX/S$.
The Deuring-Shafarevich formula, see \cite{sullivan1975} or \cite[Theorem 11,62]{hirschfeld-korchmaros-torres2008}, states that
\begin{equation}
    \label{eq2deuring}
\gamma-1={|S|}(\bar{\gamma}-1)+\sum_{i=1}^k (|S|-\ell_i)
    \end{equation}
where $\ell_1,\ldots,\ell_k$ are the sizes of the short orbits of $S$.
 If $G$ has no short orbits, that is, the Galois extension $\mathbb{K}(\cX)$ of $\mathbb{K}(\bar{\cX})$ is unramified, then
$G$ can be generated by $\bar{\gamma}$ elements by Shafarevich's theorem \cite[Theorem 2]{shafarevic1954}, whereas the largest elementary abelian subgroup of $G$ has rank at most $\bar{\gamma}$
see  \cite[Section 4.7]{pries2011}.

The following result is a special case of the classification of automorphism groups of genus $2$ curves due to Igusa \cite[Section 8]{igu} and refined in \cite[Section 1]{car}; see also \cite[Lemma 1]{sb},
and \cite{carq}.
\begin{proposition}
\label{igusa1} Let $G$ be the automorphism group of a genus $2$ curve $\cC$ in characteristic
$3$. If $3$ divides $|G|$, then either $G\cong D_{12}$, or $G\cong GL(2,3)$, and
$\cC$ is a non-singular model of an irreducible plane curve of equation $Y^2=cX^6+X^4+X^2+1$ where $c\in\mathbb{K}^*$ with $c=1$ when
$G\cong GL(2,3)$.
\end{proposition}
\begin{rem}
\label{igusa1bis}
{\em{An equivalent equation for $\cC$ in Proposition \ref{igusa1} is
\begin{equation}
\label{eq2} X(Y^3-Y)-X^2+c=0,
\end{equation}
and the linear map $(X,Y)\mapsto (X,Y+1)$ is an automorphism of $\cC$. }}
\end{rem}
{}From group theory we use the following results, see  for instance \cite{GT}.
\begin{proposition}
\label{class27}
Let $G$ be a non-abelian group of order $27$. Then either $G\cong C_9\rtimes C_3$ or $G\cong UT(3,3)$ where  $UT(r,3)$ denotes the group of upper-triangular unipotent $r\times r$ matrices over the field with three elements.
\end{proposition}
{}From Projective geometry, the following known result is used. For the  sake of completeness we include a proof.
\begin{lemma}\label{flag}
In the $r$-dimensional projective space  $PG(r,\mathbb{K})$ over an algebraically closed field $\mathbb{K}$ of characteristic $p$, let $S$ be a finite $p$-subgroup of $PGL(r+1,\mathbb K)$. If $r\geq 2$ then $S$ preserves a
flag $$\Pi_0\subset \Pi_1 \subset \ldots \subset \Pi_{r-1}$$
where $\Pi_i$ is an $i$-dimensional projective
subspace
of $PG(r,\mathbb{K})$.
\end{lemma}
\begin{proof} In $PG(2,\mathbb{K})$, any $p$-subgroup of $PGL(3,\mathbb{K})$ fixes an incident point-line pair. Hence,  the claim holds for $r=2$. We prove it for $r\geq 3$ by induction on $r$.

We first show that $S$ fixes a point of $PG(r,\mathbb{K})$. Take a non-trivial element $s$ in the center $Z(S)$ of $S$. Since $\mathbb{K}$ is algebraically closed, $s$ fixes some points of $PG(r,\mathbb{K})$. If $s$ has a unique fixed point, then $S$ itself fix that point. Therefore, we may assume that $s$ has more than one fixed point. Since $s$ has order a power of $p$, the set of its fixed points is a proper projective subspace $\Pi$ of $PG(r,\mathbb{K})$. From $s\in Z(S)$ it follows that $S$ preserves $\Pi$. By the inductive hypothesis, $S$ fixes a point in $\Pi$, and hence in
$PG(r,\mathbb{K})$.

Now, let $P\in PG(r,\mathbb{K})$ be a fixed point of $S$. The linear system of hyperplanes through $P$ is a projective space $PG(r-1,\mathbb{K})$ over $\mathbb{K}$, and $S$ acts on $PG(r-1,\mathbb{K})$, not necessarily faithfully, as a $p$-subgroup $S_1$ of $PGL(r,\mathbb{K})$. From what we have already proven, $S_1$ fixes a point $Q$ of $PG(r-1,\mathbb{K})$. Therefore, $Q$ regarded as a hyperplane $\Pi_{r-1}$ of $PG(r,\mathbb{K})$ through $P$ is preserved by $S$. From the inductive hypothesis on $r$, $S$ has an invariant flag $\Pi_0\subset \Pi_1 \subset \ldots \subset \Pi_{r-2}.$ This together with $\Pi_{r-1}$ gives an $S$-invariant flag in $PG(r,\mathbb{K})$.
\end{proof}

\section{The action of $S$ on $\cX$}
In this section, $\cX$ stands for a curve which satisfies the hypotheses of Theorem \ref{princ} but does not have the property given in (i) of Theorem \ref{princ}.
\begin{proposition}
\label{prop19feb1} The following results hold:
\begin{itemize}
\item[\rm(I)] $p=3$,
\item[\rm(II)] $S$ fixes no point in $\cX$,
\item[\rm(III)] $\gamma(\cX)>0$.
\end{itemize}
\end{proposition}
\begin{proof} (II) holds, otherwise comparing (\ref{hyp}) with (\ref{eq117feb2013}) would yield (i) of Theorem \ref{princ}. From this (III) follows by \cite[Lemma 11.129]{hirschfeld-korchmaros-torres2008}. In particular,
(\ref{naka16feb2013}) applies, and comparing (\ref{hyp}) with (\ref{naka16feb2013}) gives (I).
\end{proof}
\begin{proposition}
\label{propap=3}  $\cX$ is an ordinary curve with  $\gg(\cX)-1=\textstyle\frac{1}{3}|S|$. Moreover, $S$ has exactly two short orbits on $\cX$, both of length $\textstyle\frac{1}{3}|S|$, and if $|S|>3$ then the identity is the unique element in $S$ fixing every point of $\Omega_1\cup \Omega_2$.
\end{proposition}
\begin{proof} Let $\gg=\gg(\cX)$ and $\gamma=\gamma(\cX)$. Our hypothesis (\ref{hyp}) yields $\gamma\geq 2$ by \eqref{naka16feb2013}. Let $\bar{\gamma}$ be the $3$-rank of the quotient curve $\bar{\cX}=\cX/S$.
{}From (\ref{eq2deuring}),
\begin{equation}
\label{eq9feb2013}
\gamma -1=\bar{\gamma}|S|-|S|+\sum_{i=1}^k(|S|-\ell_i)=(\bar{\gamma}+k-1)|S|-\sum_{i=1}^k\ell_i\ge (\bar{\gamma}+\frac{2}{3}k-1)|S|,
\end{equation}
where $\ell_1,\ldots,\ell_k$ are the sizes of the short orbits of
$S$.

If no such short orbits exist, then $\gamma-1=|S|(\bar{\gamma}-1)$  whence $\bar{\gamma}>1$ by $\gamma\geq 2$.
Therefore, $|S|{\le }\gamma-1\leq \gg-1$ contradicting (\ref{hyp}).

Hence $k\geq 1$, and if $\bar{\gamma}\geq 1$ then (\ref{eq9feb2013}) yields that $|S|\leq \textstyle\frac{3}{2}(\gamma-1)\leq \textstyle\frac{3}{2}(\gg-1)$ contradicting (\ref{hyp}).
So, $\bar{\gamma}=0$, and (\ref{eq9feb2013}) together with (\ref{hyp}) imply that $k<\textstyle\frac{9}{4}$ whence $1\le k \le 2$. The case $k=1$ cannot actually occur by (\ref{eq9feb2013}).

Therefore, $\bar{\gamma}=0$ and $k=2$. Let $\Omega_1$ and $\Omega_2$ be the short orbits of $S$, and let $\ell_i=|\Omega_i|$ for $i=1,2$. Then (\ref{eq9feb2013}) reads
\begin{equation}
\label{eq9afeb2013}
\gamma-1=|S|-(\ell_1+\ell_2)
\end{equation}
whence $|S|>2(\gamma-1)=2|S|-2(\ell_1+\ell_2)$. Also, $\ell_1+\ell_2<|S|$. Write $|S|=3^h,\ell_1=3^m,\ell_2=3^r$ with $h>m\geq r$.
Then $3^m+3^r<3^h<2(3^m+3^r)$. Here $r>0$ by (II) of Proposition \ref{prop19feb1}. Therefore $3^{m-r}+1<3^{h-r}<2(3^{m-r}+1)$.
Since $h>m\geq r$, this yields $m=r$, that is, $\ell_1=\ell_2$, and $h=r+1$. In other words,  $\ell_1=\ell_2=\textstyle\frac{1}{3}|S|$. From (\ref{eq9afeb2013}), $\gamma-1=\textstyle\frac{1}{3}|S|$.

Let $\bar{\gg}$ be the genus of the quotient curve $\bar{\cX}=\cX/S$.
The Hurwitz genus formula,  (\ref{eq1}) and (\ref{eq1bis}), applied to $S$ gives
\begin{equation}
\label{eq19feb}
2\gg-2= |S|(2\bar \gg-2)+\textstyle\frac{2}{3}|S|(4+k_1+k_2)
\end{equation}
where, for a point $P_i\in \Omega_i$,  $k_i$ is the smallest non-negative integer such that $|S_{P_i}^{2+k_i}|=1$. Suppose on the contrary that $\cX$ is not an ordinary curve, that is, $\gg>\gamma$. Then $k_1+k_2\geq 1$. From (\ref{eq19feb}), $2\gg-2\ge -2|S|+\frac{10}{3}|S|=\frac{4}{3}|S|$ contradicting (\ref{hyp}).

Assume that a non-trivial element $s$ of $S$ fixes $\Omega_1\cup \Omega_2$ pointwise. From the Deuring-Shafarevic formula (\ref{eq2deuring}) applied to $\langle s \rangle$,
$$|s|=\frac{1}{3}\,|S| (2|s|-3),$$
which is only possible for $|s|=|S|=3$.
\end{proof}
{}From now on, $\Omega_1$ and $\Omega_2$ denote the short orbits of $S$ on $\cX$. By the second claim of Proposition \ref{propap=3},
\begin{equation}
\label{eq130apr2013} {\mbox{$|S_P|=3$ for $P\in \Omega_1\cup \Omega_2$}}.
\end{equation}
\begin{proposition}
\label{propbp=3} If $S$ is  abelian then either $|S|=3$, or $|S|=9$ and $S$ is elementary abelian.
\end{proposition}
\begin{proof} For a point $P\in \Omega_1$, the stabilizer $S_P$ of $P$ in $S$ is a subgroup of order $3$. Since $S$ is abelian, $S_P$ fixes
every point in $\Omega_1$. Let $\gamma^*$ be the $3$-rank of the quotient curve $\cX/S_P$. From (\ref{eq2deuring}) and Proposition \ref{propap=3},
$$\textstyle\frac{1}{3}|S|=\gamma-1\geq 3(\gamma^*-1)+\textstyle\frac{2}{3}|S|\geq \textstyle\frac{2}{3}|S|-3$$
whence $|S|=9$ and $\gg=\gamma=4$. Assume on the contrary that $S$ is cyclic. For a point $Q\in \Omega_2$ the stabilizer $S_Q$ is a subgroup of $S$ of order $3$. Since $S$ is cyclic, it has only one subgroup of order $3$. Therefore $S_P=S_Q$, and
$$\textstyle\frac{1}{3}|S|=\gamma-1\geq 3(\gamma^*-1)+\textstyle\frac{2}{3}|S|+\textstyle\frac{2}{3}|S|\geq \textstyle\frac{4}{3}|S|-3,$$
which implies $|S|=1$, a contradiction.
\end{proof}
\begin{proposition}
\label{propcp=3} Assume that $|S|\geq 9$. Let $N$ be a non-trivial normal subgroup of $S$. Then either $N$ is semiregular on $\cX$, or $S=N\rtimes \langle \varepsilon \rangle$ with $\varepsilon^3=1$.
\end{proposition}
\begin{proof} The assertion trivially holds for $|S|=9$. Assume that some non-trivial element in $N$ fixes point $P$. From the Hurwitz genus formula applied to $N$, we have
$\textstyle\frac{1}{3}|S|>|N|(\bar{\gg}-1)$ where $\bar{\gg}$ is the $3$-rank of the quotient curve $\bar{\cX}=\cX/N$. Let $\bar{S}$ be the automorphism group of $\bar{\cX}$ induced by $S$. Then $|\bar{S}||N|=|S|$ and hence
$\textstyle\frac{1}{3}|\bar{S}|>\bar{\gg}-1.$
Assume that $\bar{\gg}\geq 2$. Then (I) holds for $\bar{S}$. If (II) were also true for $\bar{S}$ then Proposition \ref{propap=3} would imply $|\bar{S}|=3(\gg-1)$. So, $\bar{S}$ fixes a point in $\bar{\cX}$. This remains true when $\bar{\gamma}=0$. Let $\bar{Q}\in\bar{\cX}$ be such a fixed point. Then the orbit $\mathcal{O}$ of $N$ consisting of all points of $\cX$ lying over $\bar{Q}$ is also an orbit of $S$. Since $\Omega_1$ and $\Omega_2$ are the only short orbits of $S$, this yields that either $\mathcal{O}$ coincides with one of them, say $\Omega_1$  Therefore, $|N|=\textstyle\frac{1}{3}|S|$. The stabilizer $\varepsilon$ of a point $R\in\Omega_2$ on $S$ has order $3$ and $\varepsilon\not\in N$. Therefore $S=N\rtimes \langle\varepsilon\rangle$.
We are left with the case $\bar{\gg}=\bar{\gamma}=1$. Let $\mathcal{O}_1,\ldots,\mathcal{O}_m$ be the short orbits of $N$. Since the stabilizer $N_Q$ of any point $Q\in \mathcal{O}_i$ has order $3$,
(\ref{eq2deuring}) applied to $N$ gives
$$\textstyle\frac{1}{3}|S|=|\gamma-1)=\textstyle\frac{2}{3}|N|m$$
whence $|S|=2|N|m$. But this is impossible as $|S|$ is a power of $3$.
\end{proof}
\begin{proposition}
\label{propgp=3} If $|S|>9$ then the center $Z(S)$ of $S$ is semiregular on $\cX$.
\end{proposition}
\begin{proof} Since $Z(S)$ is a normal subgroup of $S$, Proposition \ref{propcp=3} applies to $Z(S)$. The case $S=Z(S) \rtimes \langle \varepsilon \rangle$ cannot actually occur since this semidirect product would be direct and $S$ would be abelian contradicting Proposition \ref{propbp=3}.
\end{proof}
\begin{proposition}
\label{a30ag2013} Let $N$ be a normal subgroup of $G$ such that $|N|\le \frac{1}{9}|S|$. Then the quotient curve $\bar{\cX}=\cX/N$ with $\bar{S}=S/N$ and $\gg(\bar{\cX})-1=(\gg-1)/|N|$ satisfies the hypotheses of Theorem \ref{princ} but does not have the property given in {\rm(i)} of Theorem \ref{princ}.
\end{proposition}
\begin{proof} By Proposition \ref{propcp=3}, the Galois $p$-extension of $\bar{\cX}$ with Galois group $\bar{G}$ is unramified. Therefore, the Deuring-Shafarevic formula applied to $\bar{G}$ gives that   $\gg-1=|\bar{G}|(\bar{\gg}-1)$.
\end{proof}
Since the center of any $p$-group is non-trivial, a straightforward inductive argument on $|S|$ depending on Proposition \ref{a30ag2013} gives the following result.
\begin{proposition}
\label{pro12dic} If there exists a curve which satisfies the hypothesis of Theorem \ref{princ} for $|S|=3^k$ but does not have the property {\rm(i)}, then for any $1\le j <k$ there also exists a curve which satisfies the hypothesis of Theorem \ref{princ} for $|S|=3^j$ but does not have the property {\rm(i)}.
\end{proposition}
\begin{proposition}
\label{propdp=3} Let $N$ be a non-trivial normal subgroup of $S$. If the factor group $S/N$ is a abelian then either $|N|=\textstyle\frac{1}{3}|S|$ or $|N|=\textstyle\frac{1}{9}|S|$, and in the latter case, $S/N$ is an elementary abelian group.
\end{proposition}
\begin{proof} The assertion is a corollary of Propositions \ref{propcp=3} and \ref{a30ag2013}.
\end{proof}
\begin{proposition}
\label{30ag2013} If $|S|\ge 9$ then the following hold.
\begin{itemize}
\item[\rm(i)] $\Phi(S)=S'$.
\item[\rm(ii)] $|\Phi(S)|=\frac{1}{9}|S|$.
\item[\rm(iii)] $S$ contains exactly four maximal subgroups, each being a normal subgroup of $S$ of index $3$.
\item[\rm(iv)] Exactly two of the four maximal subgroups of $S$ are semiregular on $\cX$.
\item[\rm(v)] $S$ can be generated by two elements.
\end{itemize}
\end{proposition}
\begin{proof} From Proposition \ref{propdp=3}, either $|S'|=\frac{1}{3}|S|$, or  $|S'|=\frac{1}{9}|S|$. In the former case, $S$ is cyclic by \cite[Hilfssatz 7.1.b]{huppertI1967}  but this contradicts Proposition \ref{propbp=3}. Therefore, $|S'|=\frac{1}{9}|S|$.  Since $S/S'$ is elementary abelian by Proposition \ref{propdp=3}, (i) holds. From this (ii) follows.
Let $\varphi$ be the natural homomorphism $S\mapsto S/\Phi(S)$. Since every maximal subgroup of $S$ contains $\Phi(S)$, there is a one-to-one correspondence between the maximal subgroups of $S$ and the subgroups of $S/\Phi(S)$. By (ii), $S/\Phi(S)$ is an elementary abelian group of order $9$ with exactly four proper subgroups. Therefore there are exactly four maximal subgroups in $S$. Also, the subgroups of $S/\Phi(S)$ are normal, and hence each of the four maximal subgroups of $S$ is normal, as well. Furthermore, the four maximal subgroups of $S/\Phi(S)$ partition the set of non-trivial elements of $S/\Phi(S)$. Hence every element of $S\setminus\Phi(S)$ belongs to exactly one of the four maximal subgroups of $S$. Take a point $P\in \Omega_1$, and let $M_1$ be the maximal subgroup of $S$ containing $S_P$. Since $M$ is a normal subgroup of $S$ and $\Omega_1$ is an $S$-orbit, this yields that $M$ contains $S_Q$ for every $Q\in \Omega_1$.  Repeating the above argument for a point in $\Omega_2$ shows that a maximal normal subgroup contains the stabilizer of each point in $\Omega_2$. From the last claim of Proposition \ref{propap=3}, these two maximal subgroups are distinct. Therefore, the remaining two maximal subgroups are both semiregular on $\cX$.

Finally, (i) together with the Burnside fundamental theorem, \cite[Chapter III, Satz 3.15]{huppertI1967} imply that $S$ can be generated by two elements.
\end{proof}
{}From now on, the following notation is used: For $i=1,2,$ $M_i$ denotes the maximal normal subgroup of $S$ containing the stabilizer of a point of $\Omega_i$ while $M_3$ and $M_4$ stand for the semiregular maximal subgroups of $S$, respectively. Proposition \ref{propcp=3} has the following corollary.
\begin{proposition}
\label{31ago2013} Neither $M_1$ nor $M_2$ is cyclic.
\end{proposition}
\begin{proposition}
\label{b30ago2013} Every normal subgroup of $S$ whose order is at most $\frac{1}{9}|S|$ is contained in $\Phi(S)$.
\end{proposition}
\begin{proof} Let $N$ be a normal subgroup of $S$. {}From \cite[Chapter III, Hilfssatz 3.4.a]{huppertI1967}, $\Phi(S)N/N$ is a subgroup of $\Phi(S/N)$. From Propositions \ref{a30ag2013} and Proposition \ref{30ag2013} applied to $\bar{\cX}=\cX/N$, we have $|\Phi(S/N)|=\frac{1}{9}|S|/|N|$. Since $\Phi(S)/(\Phi(S)\cap N) \cong \Phi(S)N/N$, this yields $|N|\le |\Phi(S)\cap N|$. Therefore, if $|N|\leq |\Phi(S)|$ then $N$ is contained in $\Phi(S)$.
\end{proof}
\begin{proposition}
\label{a4set2013} If $M_i$ is abelian for some $1\le i\le 4$, then $S$ has maximal class.
\end{proposition}
\begin{proof} By (i) of Proposition \ref{30ag2013}, the assertion follows from an elementary result, see for instance \cite[Theorem 5.2]{xu2008}.
\end{proof}
It should be noted that $p$-groups of maximal class have intensively investigated, see \cite{huppertI1967} and \cite{berkovich2002}.
\begin{proposition}
\label{prop6febb2013} If $|S|>3$ then $\cX$ is not hyperelliptic.
\end{proposition}
\begin{proof} Since the length of any $S$-orbit in $\cX$ is divisible by $3$, the number of distinct Weierstrass points of $\cX$ is also divisible by $3$. On the other hand, a hyperelliptic curve of genus $g$ defined over a field of zero or odd characteristic has as many as $2g+2$ Weierstrass points, see \cite[Theorem 7.103]{hirschfeld-korchmaros-torres2008}. Therefore, if $\cX$ were hyperelliptic, both numbers $2g+2$ and $2g-2=\frac{2}{3}|S|$ would be divisible by $3$, a contradiction.
\end{proof}
\begin{proposition}
\label{a31ago2013}
If $|S|=27$ then $S\cong UT(3,3)$.
\end{proposition}
\begin{proof} If $|S|=27$ then either $S=C_9\rtimes C_3$, or $S$ is isomorphic to the group $UT(3,3)$ of all upper-triangular unipotent $3\times 3$ matrices over the field with three elements. Since the group $C_9\rtimes C_3$ has only one non-cyclic maximal subgroup, the assertion follows from Proposition \ref{31ago2013}.
\end{proof}
\begin{proposition}
\label{d4set2013} If $|S|\ge 81$ then $S$ is not metacyclic.
\end{proposition}
\begin{proof} We show that $S$ has a normal subgroup of index $27$. This is certainly true when $M_3$ is cyclic. Otherwise, $M_3$ can be generated by two elements, and its Frattini subgroup $\Phi(M_3)$ has index $27$ in $S$. Since $\Phi(M_3)$ is a characteristic subgroup of $M_3$ and $M_3$ is a normal subgroup of $S$, $\Phi(M_3)$ has the required property. From  Proposition \ref{a30ag2013} applied to $N=\Phi(M_3)$ shows that $\bar{S}=S/\Phi(M_3)$ is a subgroup of $\aut(\bar{\cX})$ with $\cX$ the quotient curve of $\bar{\cX}=\cX/\Phi(M_3)$. Since $\Phi(M_3)$ is semiregular on $\cX$ by Proposition \ref{propcp=3}, $\bar{\cX}$ has genus $10$. As $|\bar{S}|=27$, Proposition \ref{a31ago2013} implies that $\bar{S}\cong UT(3,3)$. In particular, $\bar{S}$ is not metacyclic. Since every factor group of a metacylic group is still metacyclic, the assertion follows.
\end{proof}

\section{Small genera}
\subsection{Cases $|S|=3,9$}
We prove that if $\cX$ satisfies the hypotheses of Theorem \ref{princ} for $|S|=3$ then (ii) holds.
For this case, our hypothesis (\ref{hyp}) yields $\gg=2$. From (\ref{eq2deuring}), every automorphism of $\aut(\cX)$ of order $3$ has two fixed points on $\cX$. Therefore, (i) of Theorem \ref{princ} cannot occur, and  the assertion follows from Proposition \ref{igusa1}.

{}From now on, $|S|=9$ and  $\cX$ is a curve satisfying the hypotheses of Theorem \ref{princ} but does not have the property given in (i) of Theorem \ref{princ}.
We prove that then (iii) holds. From Propositions  \ref{propap=3} and \ref{propbp=3}, $\cX$ is an ordinary curve of genus $\gg=4$ with  an elementary abelian subgroup $S$ of $\aut(\cX)$ of order $9$.

Let $N$ be the kernel of the permutation representation of $S$ on $\Omega_1\cup \Omega_2$. If $N$ is not trivial then it has order $3$, and the Hurwitz genus formula (\ref{eq1}) and (\ref{eq1bis}) applied to $N$ gives $6=2(\gg-1)\ge 6(\bar{\gg}-1)+24$.
Therefore $S$ acts on $\Omega_1\cup\Omega_2$ faithfully.

By Proposition \ref{prop6febb2013}, $\cX$ is assumed to be a canonical curve embedded in $PG(3,\mathbb{K})$. Then $S$ extends to a subgroup of $PG(3,\mathbb{K})$
which preserves $\cX$ and acts on $\cX$ faithfully.

According to Lemma \ref{flag}, choose the projective coordinate system $(X_0:X_1:X_2:X_3)$ in $PG(3,\mathbb{K})$ in such a way that 
$S$ preserves the canonical flag $$P_0\subset\Pi_1\subset \Pi_2$$
where $P_0=(1:0:0:0)$, $\Pi_1$ is the line through $P_0$ and $P_1=(0:1:0:0)$ while $\Pi_2$ is the plane of equation $X_3=0$. Here $P_0\not\in \cX$, since $S$ fixes no point in $\cX$. Moreover,
$\Pi_2\cap\cX=\Omega_1\cup \Omega_2$. In fact, for any point $R\in \Pi_2\cap \cX$,  Proposition \ref{propap=3} implies that the $S$-orbit of $R$ has size $9$ unless $R\in \Omega_1\cup \Omega_2$. On the other hand $S$ preserves $\Pi_2\cap \cX$, and this implies that the $S$-orbit of $R$ cannot exceed $6$.
\begin{lemma}
\label{lem7feba2013} Both $\Omega_1$ and $\Omega_2$ consist of three collinear points.
\end{lemma}
\begin{proof} Assume on the contrary that $\Omega_1$ is a triangle. Take $g\in S$ such that $g$ fixes each vertex of $\Omega_1.$ Since $g$ is a projectivity of $PG(3,\mathbb{K})$ it fixes $\Pi_2$ pointwise. As $\Pi_2$ also contains  $\Omega_2,$ $g$ must fix $\Omega_2$ pointwise. But this is impossible as $S$ acts on $\cX$ faithfully.
\end{proof}
As a corollary, $I(R,\cX\cap \Pi_2)=1$ for every point $R\in\Omega_1\cup \Omega_2$.  For $i=1,2$, let $r_i$ denote the line containing $\Omega_i$. Their common point is fixed by $S$, and may be chosen for $P_0$. Let $M$ be the subgroup of $S$ which preserves every line through $P_0$. Since $\deg\, \cX=6$, no line meets $\cX$ in more than six distinct points. Therefore, either $|M|=1$ or $|M|=3$. In the latter case, $M$ is an elation group of order $3$ with center $P_0$. If $\Delta$ is its axis then
every point in $\Delta\cap \cX$ is fixed by $M$. Therefore $\Delta$ is not $\Pi_2$ and contains either $r_1$ or $r_2$. Since $S$ is abelian, it  preserves $\Delta$ and hence every plane through $r_1$. But then every $S$-orbit has length at most $3$. A contradiction with Proposition \ref{propap=3}. Hence $M$ is trivial.

Since $P_0\not \in \cX$, the linear system $\Sigma$ of all planes through $P_0$ cuts out on $\cX$ a linear series without fixed point. Therefore this effective linear series has dimension $2$ and degree $6$, and is denoted by $g_2^6$.

It might happen that $g_2^6$ is composed of an involution, and we investigate such a possibility. From \cite[Section 7.4]{hirschfeld-korchmaros-torres2008}, there is a curve ${\cZ}$ whose function field $\mathbb{K}({\cX})$ is an $S$-invariant proper subfield of $\mathbb{K}(\cX)$. Since no non-trivial element in $S$ fixes every line through $P_0$ and hence every plane through $P_0$, $S$ acts on ${\cZ}$ faithfully.  As the genus of $\cZ $ is less than $4$,  applying \eqref{naka16feb2013} to $\cZ$ gives   $\gamma(\cZ)=0$. Therefore, every non-trivial element in $S$ has a unique fixed point $\bar{T}$, see  \cite[Lemma 11.129]{hirschfeld-korchmaros-torres2008}.
{}From this, the support of the divisor of $K(\cX)$ lying over $\bar{T}$ contains the points in $\Omega_1\cup \Omega_2$. Therefore, the line through $P_0$ and a point in $\Omega_1\cup \Omega_2$ must contain all the points in $\Omega_1\cup \Omega_2$. But this would imply that $r_1=r_2$, a contradiction.

Therefore, $g_2^6$ is simple and without fixed point. The projection of $\cX$ from $P_0$ is an irreducible plane curve $\cC$ of degree $6$ and genus $4$ with two triple points $R_1$ and $R_2$ arising from $\Omega_1$ and $\Omega_2$, respectively. Here $\cC$ and $\cX$ are birationally equivalent, and $S$ is a subgroup of $PGL(3,\mathbb{K})$ preserving $\cC$.  For $i=1,2$, a non-trivial projectivity $s_i\in S$ fixing $\Omega_i$ pointwise acts on $\cC$ fixing the point $R_i$.

Choose the projective coordinate system $(X_0:X_1:X_2)$ in $PG(2,\mathbb{K})$ so that $R_1=(0:0:1)$ and $R_2=(0:1:0)$. In affine coordinates $(X,Y)$ with $X=X_1/X_0,\,Y=X_2/X_0$, an equation of $\cC$ is  $f=0$ with an irreducible polynomial $f\in \mathbb{K}[X,Y]$ of degree six. W.l.o.g. the origin $O=(0,0)$ is the common point of two tangents to $\cC$, say $t_1$ at $R_1$ and $t_2$ at $R_2$. Furthermore,  $s_1(O)=(\lambda,0)$, $s_2(O)=(0,\mu)$ with $\lambda,\mu\in \mathbb{K}^*$, and $\lambda=\mu=1$ may be assumed.
Thus $s_1:\,(X,Y)\mapsto (X+1,Y)$ and $s_2:\,(X,Y)\mapsto (X,Y+1)$. Hence
\begin{equation}
\label{eq7feb2013} f(X+1,Y)=f(X,Y),\quad f(X,Y+1)=f(X,Y).
\end{equation}
Since $R_1$ is a triple point of $\cC$, there exist $h_0,h_1,h_2,h_3\in \mathbb{K}[Y]$ such that $$f(X,Y)=h_3X^3+ h_2X^2+h_1X+h_0=0,$$
where $\deg\, h_0\leq 2$ by the particular choice of $t_2$.
{}From this and (\ref{eq7feb2013}), the polynomial
\begin{equation}
\label{eq7febb2013}
f(X+1,Y)-f(X,Y)=h_3-h_2X+h_2+h_1
\end{equation}
 vanishes at every affine point of $\cC$. Since $\cC$ is not rational, this is only possible when (\ref{eq7febb2013}) is the zero polynomial, that is,   $h_2=
h_3+h_1=0.$ Thus $f(X,Y)=h_3(X^3-X)+h_0$. The second mixed partial derivate is $f_{X,Y}=-dh_3/dY$. Similarly, as $R_2$ is a triple point of $\cC$ there exist $k_0,k_3\in \mathbb{K}[X]$ with $\deg\,k_0\leq 2$ such that $f(Y,X)=k_3(Y^3-Y)+k_0$. Since $f_{X,Y}=f_{Y,X}$, this yields $dh_3/dY=dk_3/dX$,  whence $dh_3/dY$ and $dk_3/dX$  both have degree $0$. Thus
$$h_3=c_3Y^3+c_1Y+c_0, \quad   k_3=d_3X^3+d_1X+d_0$$
where $c_0,c_1,c_3,d_0,d_1,d_3\in \mathbb{K}.$ Therefore
$$(c_3Y^3+c_1Y+c_0)(X^3-X)+h_0=(d_3X^3+d_1X+d_0)(Y^3-Y)+k_0.$$
Comparison of the coefficients of $X^3$  shows that $c_3=-c_1,c_0=0.$ Similarly, $d_3=-d_1,d_0=0.$ Thus
$$f(X,Y)=c(X^3-X)(Y^3-Y)+h_0=d(Y^3-Y)(X^3-X)+k_0$$
where $c,d\in \mathbb{K}.$ From this
$(c-d)(X^3-X)(Y^3-Y)=k_0-h_0$ whence $c=d$ and $k_0=h_0=u$ with $u\in \mathbb{K}^*$. Therefore
$$f(X,Y)=(X^3-X)(Y^3-Y)+c=0$$
where $c\in \mathbb{K}^*$. This ends the proof of  Theorem \ref{princ} for $|S|=9$.
\subsection{Case $|S|=27$}
\label{S27}
We exhibit an explicit example.  Let $\cX$ be a non-singular model of the irreducible curve $\cY$ embedded in $PG(3,\mathbb{K})$ defined by the
affine equations
\begin{itemize}
\item[(i)] $(X^3-X)(Y^3-Y)-1=0$;
\item[(ii)] $Z^3-Z+X^3Y-XY^3=0$.
\end{itemize}
A straightforward Magma computation shows that $\gg(\cX)=\gamma(\cX)=10$. Moreover, both maps
$$ g:\,(X,Y,Z)\mapsto (X+1,Y,Z+Y)\quad h:\,(X,Y,Z)\mapsto (X,Y-1,Z+X)$$
are in $\aut(\cX)$. They generate a non-abelian group $S$ of order $27$ and exponent $3$. Therefore $S\cong UT(3,3)$. Actually, $\aut(\cX)$ is larger than $S$ since it also contains $r:\,(X,Y,Z)\mapsto (Y,X,Z)$. This shows that $\cX$ satisfies the hypotheses of Theorem \ref{princ} with $S\cong UT(3,3)$.

\subsection{Case $|S|=81$}
\begin{proposition}
\label{prop20bapr2013}
For $|S|=81$ there are only two possibilities for $S$, namely
\begin{itemize}
\item[{\rm{(a)}}] $S\cong S(81,7)$ where $S(81,7)=C_3 \wr C_3$ is the Sylow $3$-subgroup  of the symmetric group of degree $9$, moreover $M_1\cong C_3\times C_3 \times C_3$, $M_2\cong UT(3,3)$, $M_3\cong M_4\cong C_9 \rtimes C_3$.
\item[{\rm{(b)}}] $S\cong S(81,9)=\langle a,b,c |a^9=b^3=c^3=1,ab=ba,cac^{-1}=ab^{-1},cbc^{-1}=a^3b\rangle$ with exactly $62$ elements of order $3$; moreover $M_1\cong M_2\cong M_3 \cong UT(3,3)$, $M_4\cong C_9\times C_3$.
\end{itemize}
\end{proposition}
\begin{proof} There exist exactly seven groups of order $81$ generated by two elements, namely $S(81,i)$ with $i=1,\ldots,7$, and each of them has an abelian normal subgroup of index $3$. By Proposition \ref{a4set2013}, $S$ is of maximal class. There are four pairwise non-isomorphic groups of order $81$ and maximal class, namely (a), (b) and
\begin{itemize}
\item[(c)] $S(81,8)\cong\langle a,b,c |a^9=b^3=c^3=1,ab=ba,cac^{-1}=ab,cbc^{-1}=a^3b\rangle$ with $26$ elements of order $3$;
\item[(d)] $S(81,10)\cong (C_9\rtimes C_3)\rtimes C_3$  with  $8$ elements of order $3$.
 \end{itemize}
One of the four maximal normal subgroups of $S(81,8)$ is isomorphic to $U(3,3)$ and hence it contains all elements of order $3$. On the other hand, (iv) of Proposition \ref{30ag2013} yields that
two of the maximal normal subgroups of $S$, namely $M_1$ and $M_2$, have non-trivial $1$-point stabilizer in $\Omega_1$ and $\Omega_2$, respectively. Hence, both must have an element of order $3$ not contained in $\Phi(S)$. Since $M_1\cap M_2=\Phi(S)$, these elements are not in the same maximal normal subgroup. This contradiction shows that (c) cannot actually occur in our situation. Regarding $S(81,10)$, all elements of order $3$ lie in $\Phi(S)$ as $\Phi(S)$ is an elementary abelian group of order $9$. But this is impossible in our situation since $M_1$ must have an element of order $3$ not in $\Phi(S)$   by Propositions \ref{30ag2013} and \ref{b30ago2013}.
\end{proof}
We exhibit an explicit example. Let $\cX$ be a non-singular model of the irreducible curve $\cY$ embedded in $PG(4,\mathbb{K})$ defined by the
affine equations
\begin{itemize}
\item[(i)] $(X^3-X)(Y^3-Y)-1=0$;
\item[(ii)] $X(U^3-U)-1=0$;
\item[(iii)] $(X+1)(V^3-V)-1=0.$
\end{itemize}
A straightforward Magma computation shows that $\gg(\cX)=\gamma(\cX)=28$.  Furthermore, each of the following four maps
$$
\begin{array}{lll}
&g_1:\,(X,Y,U,V)\mapsto (X,Y+1,U,V), &  g_2:\,(X,Y,U,V)\mapsto (X,Y,U+1,V),\\
&g_3:\,(X,Y,U,V)\mapsto (X,Y,U,V+1), &  g_4:\,(X,Y,U,V)\mapsto(X+1,Y,V,-Y-U-V),
\end{array}
$$
are in $\aut(\cX)$. They generate a non-abelian group $S$ of order $81$ isomorphic to $S(81,7)$.
\subsection{Case $|S|=243$}
\begin{proposition}
\label{b1sep2013} If
 $|S|=243$ and $S$ has a maximal abelian subgroup, then there are only two possibilities for $S$, namely
\begin{itemize}
\item[\rm(i)] $S$ is isomorphic to the unique group $S(243,26)$ of order $243$ with $170$ elements of order $3$, moreover $M_2\cong M_3\cong M_4\cong S(81,9)$, and $M_1\cong C_9\times C_9$,
\item[\rm(ii)] $S$ is isomorphic to the unique group $S(243,28)$ of order $243$ with $116$ elements of order $3$, moreover $M_1\cong M_2\cong S(81,9)$ while $M_3\cong S(81,4)$, and $M_4\cong S(81,10)$.
\end{itemize}
\end{proposition}
\begin{proof} There exist exactly six pairwise non-isomorphic groups of order $81$ and maximal class, namely (i), (ii) and $S(243,25)$ with $62$ elements of order $3$; $S(243,27)$ with $8$ elements of order $3$; $S(243,29)$ with $8$ elements of order $3$; $S(243,30)$ with $62$ elements of order $3$.

One of the four maximal normal subgroups of $S(243,28)$ (and of $S(243,30)$) is isomorphic to $S(81,8)$ and hence it contains all elements of order $3$. The argument in the proof of Proposition \ref{prop20bapr2013} ruling out possibility (c) also works in this case. Therefore, neither $S\cong S(243,25)$ nor $S\cong S(243,28)$ is possible. Regarding $S(243,27)$ and $S(243,29)$, we may use the argument from the proof of Proposition \ref{prop20bapr2013} that ruled out possibility (d). Therefore, $S\cong S(243,25)$ and $S\cong S(243,28)$ cannot occur in our situation.
\end{proof}
\section{Infinite Family of Examples} Let $\cC$ be a general curve of genus $2$ as given in Remark \ref{igusa1} with function field $F=\mathbb{K}(\cX)=\mathbb{K}(x,y)$ where
\begin{equation}
\label{eq2dic10} x(y^3-y)-x^2+c=0.
\end{equation}
For a positive integer $N$, let $F_N$ be the largest unramified abelian
extension of $F$ of exponent $N$; that is, $F_N|F$ has the following three properties:
\begin{itemize}
\item[(i)] $F_N|F$ is an unramified Galois extension;
\item[(ii)] $F_N$ is generated by all function fields which are cyclic unramified extensions of $F$ of degree $p^N$,
\item[(iii)] ${\rm{Gal}}(F_N|F)$ is abelian and  $u^{3^N}=1$ for every element $u\in {\rm{Gal}}(F_N|F)$.
\end{itemize}
{}From classical results  due to Schmid and Witt \cite{schmidwitt1938}, $\deg(F_N|F)=3^{2N}$ and ${\rm{Gal}}(F_N|F)$ is the direct product of two copies of the cyclic group of order $3^N$.
Let $\cX$ be the curve such that $F_N=\mathbb{K}(\cX)$. Since $F_N$ is an unramified extension of $F$, from the Deuring-Shafarevic formula (\ref{eq2deuring}) we have that $\gamma(\cX)-1=3^{2N}(2-1)=3^{2N}$. Our aim is to prove that $\aut(\cX)$ contains a $3$-group of order $3^{2N+1}$.

Let $\mathbb{K}(x)$ be rational the subfield of $F$ generated by $x$. Obviously, $\mathbb{K}(x)$ is a subfield of $F_N$ and we are going to consider the Galois closure $M$ of $F_N|\mathbb{K}(x)$. Let  $M=\mathbb{K}(\cY)$ where $\cY$ is an algebraic curve defined over $\mathbb{K}$. Take any $\mu\in {\rm{Gal}}(M|\mathbb{K}(x))$. Then $\mu$ is a $\mathbb{K}$-automorphism of $\cY$ fixing $x$. Let $v=\mu(y)$. Since $\mu(x(y^3-y)-x^2+c)=x(v^3-v)-x^2+c$, from (\ref{eq2dic10})
$$x(v^3-v)-x^2+c=0.$$  This together with (\ref{eq2dic10}) yield that either $v=y$ or $v=y\pm 1$. In both cases $v\in F$. Therefore, ${\rm{Gal}}(M|\mathbb{K}(x))$ viewed as a subgroup $S$ of $\aut(\cY)$ preserves
$F$. From the definition of $F_N$, this implies that $S$ also preserves $F_N$, that is $S$ is a subgroup of $\aut(\cX)$. Since $F_N\subseteq M$, we have that
$$[F_N:\mathbb{K}(x)]\leq [M:\mathbb{K}(x)]=|S|.$$ Furthermore, $$[F_N:\mathbb{K}(x)]=[F_N:F][F:\mathbb{K}(x)]=3^{2N}3=3^{2N+1}.$$
On the other hand $|S|>3^{2N+1}$ is impossible by the Nakajima's bound.

As a corollary, $\aut(\cX)$ has a subgroup $S$ so that the Hasse-Witt invariant $\gamma(\cX)$ is equal to $|S|/3$.

Our construction together with Proposition \ref{pro12dic} provides a curve of type (ii) in Theorem \ref{princ}, for every power of $3$.

In particular, if $N=1$ then the resulting
curve is isomorphic to that given in Subsection \ref{pro12dic}. For $N=2$, we have $S\cong
S(243,26)$ and case (i) in Proposition \ref{b1sep2013} occurs;
%
%
furthermore, $S/Z(S)$ is isomorphic to $S(81,9)$, and the quotient curve $\cX/Z(S)$
provides an example for type (b) in Proposition \ref{prop20bapr2013}.

\vspace{0,5cm}\noindent {\em Authors' addresses}:

\vspace{0.2 cm} \noindent Massimo GIULIETTI \\
Dipartimento di Matematica e Informatica
\\ Universit\`a degli Studi di Perugia \\ Via Vanvitelli, 1 \\
06123 Perugia
(Italy).\\
 E--mail: {\tt giuliet@dipmat.unipg.it}

\vspace{0.2cm}\noindent G\'abor KORCHM\'AROS\\ Dipartimento di
Matematica\\ Universit\`a della Basilicata\\ Contrada Macchia
Romana\\ 85100 Potenza (Italy).\\E--mail: {\tt
gabor.korchmaros@unibas.it }

    \end{document}